\documentclass{article}

\newtheorem{theorem}{Theorem}

\newtheorem{corollary}[theorem]{Corollary}

\newtheorem{lemma}[theorem]{Lemma}

\newtheorem{proposition}[theorem]{Proposition}

\newenvironment{proof}[1][Proof]{\noindent\textbf{#1.} }{\ \rule{0.5em}{0.5em}}
\input{tcilatex}

\begin{document}

\begin{center}
{\Large Regular Variation and Raabe}

Christopher N.B.\ Hammond$^{1}$ and Edward Omey$^{2}$

\bigskip

$^{1}$Department of Mathematics and Statistics,

Connecticut College, New London, CT 06320

E-mail address: cnham@conncoll.edu

$^{2}$MEES, KU\ Leuven, Warmoesberg 26,

1000 Brussels, Belgium

E-mail address: edward.omey@kuleuven.be
\end{center}

\bigskip 

\begin{center}
\textbf{Abstract}
\end{center}

There are many tests for determining the convergence or divergence of
series. The test of Raabe and the test of Betrand are relatively unknown and
do not appear in most classical courses of analysis. Also, the link between
these tests and regular variation is seldomly made. In this paper we offer a
unified approach to some of the classical tests from a point of view of
regular varying sequences.

\bigskip 



\bigskip 

\section{Introduction}

The aim of this paper is to offer a unified approach to some classical tests
of convergence (and divergence) of series with real or with positive terms.
In the paper among others we use Raabe's test that has been introduced by
Joseph Ludwig Raabe in 1832 and we make a link to regularly varying
sequences that have basically been introduced by Jovan Karamata in the
1930's.

\section{Regularly varying sequences}

\subsection{Definition}

A sequence $(c_{n})$ of positive numbers is regularly varying with index $%
\alpha \in \Re $ if it satisfies:%
\[
\lim_{n\rightarrow \infty }\frac{c_{\left[ nx\right] }}{c_{n}}=x^{\alpha
},\forall x>0\text{.} 
\]

We call such a sequence a \textit{regularly varying sequence} and we use the
notation: $(c_{n})\in RS_{\alpha }$.

\bigskip

There is a close relationship between regularly varying sequences and
regularly varying functions. A positive and measurable function $f(.)$ is
regularly varying with index $\alpha \in \Re $ if it satisfies:%
\[
\lim_{t\rightarrow \infty }\frac{f(tx)}{f(t)}=x^{\alpha },\forall x>0\text{.}
\]

Notation: $f\in RV_{\alpha }$.

Regularly functions were introduced by J.\ Karamata (1930, 1933) in
connection with Tauberian theorems. The relation between $RV$ and $RS$ was
established by Bojanic and Seneta (1973). They proved the following result:%
\[
(c_{n})\in RS_{\alpha }\text{ if and only if }f(x):=c_{\left[ x\right] }\in
RV_{\alpha }\text{.} 
\]

This result implies that all properties of regularly varying functions can
be reformulated in terms of regularly varying sequences. In the next section
we list some important properties of regularly varying sequences, see
Bojanic and Seneta (1973), Bingham et al. (1987, section 1.9), Geluk and de
Haan (1987, section I.3)

\subsection{Properties}

The following representation theorem holds.

\begin{theorem}
$(c_{n})\in RS_{\alpha }$ if and only if $c_{n}$ may be written as%
\begin{equation}
c_{n}=\alpha (n)n^{\alpha }\exp \sum_{k=1}^{n}k^{-1}\delta _{k}\text{,}
\label{1}
\end{equation}%
where $\alpha (n)\rightarrow C>0$ and $\delta _{n}\rightarrow 0$ as $%
n\rightarrow \infty $.
\end{theorem}

From (1) we obtain the following corollary.

\begin{corollary}
(i) If $(c_{n})\in RS_{\alpha }$, then $\log c_{n}/\log n\rightarrow \alpha $%
.

(ii) If $(c_{n})\in RS_{\alpha }$, and $\alpha <0$, then $c_{n}\rightarrow 0$%
.

(iii) If $(c_{n})\in RS_{\alpha }$, and $\alpha >0$, then $c_{n}\rightarrow
\infty $.
\end{corollary}

\bigskip

Using Theorem 1 we also have the following algebraic properties.

\begin{corollary}
Suppose that $(c_{n})\in RS_{\alpha }$ and $(b_{n})\in RS_{\beta }$.

(i) We have $(c_{n}b_{n})\in RS_{\alpha +\beta }$ and $(c_{n}/b_{n})\in
RS_{\alpha -\beta }$.

(ii) For any real number $r$ we have $(c_{n}^{r})\in RS_{\alpha r}$.

(iii) We have $(c_{n}+b_{n})\in RS_{\max (\alpha ,\beta })$.

(iv) If $\beta <\alpha $, then $(c_{n}-b_{n})\sim c_{n}\in RS_{\alpha }$.
\end{corollary}

\bigskip

When we consider sums, we can use the following result known as Karamata's
theorem.

\begin{theorem}
Suppose that $(c_{n})\in RS_{\alpha }$.

(i) If $\alpha >-1$, then $\sum_{k=1}^{n}c_{k}\rightarrow \infty $ and $%
\sum_{k=1}^{n}c_{k}\sim nc_{n}/(\alpha +1)$.

(ii) If $\alpha <-1$, then $\sum_{k=1}^{\infty }c_{k}<\infty $ and $%
\sum_{k=n}^{\infty }c_{k}\sim nc_{n}/(-1-\alpha )$.

(iii) If $\alpha =-1$, then $nc_{n}/\sum_{k=1}^{n}c_{k}\rightarrow 0$.

(iv) If $\alpha =-1$ and $\sum_{k=1}^{\infty }c_{k}<\infty $, then $%
nc_{n}/\sum_{k=n}^{\infty }c_{k}\rightarrow 0$.
\end{theorem}

\bigskip

Note that the cases $\alpha <-1$ and $\alpha >-1$ allow to decide about
convergence or divergence of the series $\sum_{k=1}^{\infty }c_{k}$. If $%
\alpha =-1$, no such conclusion is possible. Both divergence and convergence
appear as the examples $c_{n}=1/(n\log (n+1)$ and $c_{n}=1/n(\log n)^{2}$
show. Karamata's theorem also gives an estimate of $\sum_{k=1}^{n}c_{k}$ in
the case of divergence or $\sum_{k=n}^{\infty }c_{k}$ in the case of
convergence.

In (1) and in the previous theoem, we assumed that $c_{\left[ nx\right]
}/c_{n}$ converges to a limit for all $x>1$. Now we consider a case where $%
c_{2n}/c_{n}$ is bounded. In Section 4 we provide additional results of this
type.

\begin{lemma}
Suppose that $c_{n+1}/c_{n}\rightarrow 1$ and that $a\leq c_{2n}/c_{n}\leq b$
for $n\geq n%
{{}^\circ}%
$.

(i) If $b<1/2$, then $\sum_{i=1}^{\infty }c_{i}<\infty $.

(ii) If $a>1/2$, then $\sum_{i=1}^{\infty }c_{i}=\infty $.
\end{lemma}

\begin{proof}
Choose $\epsilon >0$ and $N\geq n%
{{}^\circ}%
$ so that $(1-\epsilon )c_{n}\leq c_{n+1}\leq (1+\epsilon )c_{n},\forall
n\geq N$. Using $c_{2n+1}/c_{n}=(c_{2n+1}/c_{2n})(c_{2n}/c_{n})$, we find
that%
\[
a\leq \frac{c_{2n}}{c_{n}}\leq b,a(1-\epsilon )\leq \frac{c_{2n+1}}{c_{n}}%
\leq b(1+\epsilon ),\forall n\geq N 
\]%
Now consider 
\[
S_{k}=(c_{N2^{k}}+c_{N2^{k}+1})+(c_{N2^{k}+2}+c_{N2^{k}+3})+...+(c_{N2^{k+1}-2}+c_{N2^{k+1}-1})%
\text{,} 
\]%
for $k=0,1,...$. Clearly for $k\geq 1$ we have 
\begin{eqnarray*}
S_{k} &\leq &(b+b(1+\epsilon ))c_{N2^{k-1}}+(b+b(1+\epsilon
))c_{N2^{k-1}+1}+...+(b+b(1+\epsilon ))c_{N2^{k}-1} \\
&=&b(2+\epsilon )S_{k-1}
\end{eqnarray*}%
Proceeding by induction, we have that%
\[
S_{k}\leq b^{k}(2+\epsilon )^{k}S_{0}, 
\]%
for $k=0,1,...$. If $b(2+\epsilon )<1$, it follows that $\sum_{k=0}^{\infty
}S_{k}=\sum_{i=N}^{\infty }c_{i}<\infty $. In a similar way we have 
\[
S_{k}\geq a^{k}(2-\epsilon )^{k}S_{0}\text{.} 
\]%
If $a(2-\epsilon )>1$, it follows that $\sum_{k=0}^{\infty
}S_{k}=\sum_{i=N}^{\infty }c_{i}=\infty $.
\end{proof}

\bigskip 

In the next result we provide a link to Raabe's type of sequences,cf.
Section 3 below. Observe that (1) implies that $c_{n}\sim b_{n}:=Cn^{\alpha
}\exp \sum_{k=1}^{n}k^{-1}\delta _{k}$.\ We have the following result.

\begin{theorem}
$(c_{n})\in RS_{\alpha }$ if and only if there is a sequence $(b_{n})$ of
positive numbers so that $c_{n}\sim b_{n}$ and 
\begin{equation}
\lim_{n\rightarrow \infty }n(\frac{b_{n+1}}{b_{n}}-1)=\alpha \text{,}
\label{2}
\end{equation}%
or equivalently%
\begin{equation}
\lim_{n\rightarrow \infty }n\log \frac{b_{n+1}}{b_{n}}=\alpha \text{.}
\label{3}
\end{equation}
\end{theorem}

\begin{proof}
Using the representation (1), we have $c_{n}\sim b_{n}:=Cn^{\alpha }\exp
\sum_{k=1}^{n}k^{-1}\delta _{k}$. Using%
\[
n\log \frac{b_{n+1}}{b_{n}}=\alpha n\log (1+\frac{1}{n})+\frac{n}{n+1}\delta
_{n+1}\text{,} 
\]%
we obtain (3) and (3). Conversely, suppose that (3) holds. Define $\alpha
(n)=n\log (b_{n+1}/b_{n}),n\geq k$. We have $\alpha (n)\rightarrow \alpha $
and%
\[
\log \frac{b_{n+1}}{b_{n}}=\frac{\alpha (n)}{n}\text{.} 
\]%
It follows that for $m=k+1,k+2,...$ we have%
\[
\log \frac{b_{m}}{b_{k}}=\sum_{i=k}^{m-1}\log \frac{b_{i+1}}{b_{i}}%
=\sum_{i=k}^{m-1}\frac{\alpha (i)}{i} 
\]%
and 
\[
\log \frac{b_{m}}{b_{k}}-\alpha \log m=\alpha (\sum_{i=k}^{m-1}\frac{1}{i}%
-\log m)+\sum_{i=k}^{m-1}\frac{\alpha (i)-\alpha }{i}\text{.} 
\]%
Note that as $m\rightarrow \infty $, $E(m):=\sum_{i=k}^{m-1}\frac{1}{i}-\log
m$ converges to a finite constant. We see that%
\[
b_{m}=m^{\alpha }b_{k}\exp E(m)+\sum_{i=k}^{m-1}\frac{\delta (i)}{i},m\geq k%
\text{,} 
\]%
where $\delta (n)\rightarrow 0$. This gives representation (1).
\end{proof}

\bigskip

Remarks.

1) From (2) we have that $(b_{n+1}-b_{n})\sim \alpha b_{n}/n$. If $\alpha >0$
it follows that $b_{n}$ is an increasing sequence.\ If $\alpha <0$, then $%
b_{n}$ is a decreasing sequence. Hence if $(c_{n})\in RS_{\alpha },\alpha
\neq 0$, then $c_{n}\sim b_{n}$ where $b_{n}$ is a monotonic sequence.

2) In the case of $f\in RV_{\alpha }$, it is always possible (cf. Bingham et
al. 1987, Theorem 1.8.2) to find a function $g\in RV_{\alpha }$ so that $%
f(x)\sim g(x)$ and $xg^{\prime }(x)/g(x\rightarrow \alpha $ as $x\rightarrow
\infty $. This is the function-analogue to (2).

\section{Raabe}

\subsection{Raabe's test}

In Raabe's test we assume that 
\begin{equation}
\lim_{n\rightarrow \infty }n(\left\vert \frac{a_{n+1}}{a_{n}}\right\vert
-1)=\alpha \text{.}  \label{4}
\end{equation}%
Clearly (4) implies that $(\left\vert a_{n}\right\vert )\in RS_{\alpha }$.
Conversely, Theorem 6 shows that any regularly varying sequence is
asymptotically equal to a sequence for which (4) holds.

Using the results of the previous section, we immediately have the
following\ theorem that contains Raabe's test.

\begin{theorem}
Suppose that (4) holds.

(i) If $\alpha <0$, then $a_{n}\rightarrow 0$ while if $\alpha >0$, we have $%
\left\vert a_{n}\right\vert \rightarrow \infty $.

(ii) If $\alpha <-1$, then $\sum_{k=1}^{\infty }\left\vert a_{k}\right\vert
<\infty $ and $\sum_{k=n}^{\infty }\left\vert a_{k}\right\vert \sim
n\left\vert a_{n}\right\vert /(-1-\alpha )$.

(iii) If $\alpha >-1$, then $\sum_{k=1}^{n}\left\vert a_{k}\right\vert
\rightarrow \infty $ and $\sum_{k=1}^{n}\left\vert a_{k}\right\vert \sim
n\left\vert a_{n}\right\vert /(\alpha +1)$

(iv) If $\alpha =-1$, then $na_{n}/\sum_{k=1}^{n}\left\vert a_{k}\right\vert
\rightarrow 0$

(v) If $\alpha =-1$ and $\sum_{k=1}^{\infty }\left\vert a_{k}\right\vert
<\infty $, then $na_{n}/\sum_{k=n}^{\infty }\left\vert a_{k}\right\vert
\rightarrow 0$.
\end{theorem}

\textbf{Examples.}

The example $a_{n}=(-1)^{n}\log (n+1)$ shows that $\alpha =0$ and $%
\left\vert a_{n}\right\vert \rightarrow \infty $.

The example $a_{n}=(-1)^{n}/\log (n+1)$ show that $\alpha =0$ and $%
a_{n}\rightarrow 0$.

The examples $a_{n}=(-1)^{n}$ and $a_{n}=(-1)^{n}n\left\vert \sin
1/n\right\vert $ show that $\alpha =0$ and $\left\vert a_{n}\right\vert
\rightarrow 1$.

The examples $a_{n}=(-1)^{n}n^{-1}\log (n+1)$ resp. $a_{n}=(-1)^{n}/n(\log
(n+1))^{2}$ show that in the case $\alpha =-1$ we can have (iv) or (v). 

\bigskip 

We need extra information to decide which case happens. If $\alpha =-1$,
relation (4) alone is not enough to decide about convergence or divergence
of $\sum_{k=1}^{n}\left\vert a_{k}\right\vert $.

\textbf{Remarks}.

1) If (4) holds, then for all $k=0,1,...$ and $r=1,2,...$ we have%
\[
\lim_{n\rightarrow \infty }(n+k)(\left\vert \frac{a_{n+r}}{a_{n}}\right\vert
-1)=\lim_{n\rightarrow \infty }(n+k)\log \frac{\left\vert a_{n+r}\right\vert 
}{\left\vert a_{n}\right\vert }=r\alpha \text{.}
\]

2) Raabe (1934) also proved the following result. Suppose that (4) holds and
that $0\leq a_{n}\rightarrow 0$. Let $f(.)$ denote a function so that $%
xf^{\prime }(x)/f(x)\rightarrow \beta $ as $x\rightarrow 0$. Then we have%
\[
n(\frac{f(a_{n+1})}{f(a_{n})}-1)\rightarrow \alpha \beta \text{.}
\]%
A similar result holds if $0\leq a_{n}\rightarrow \infty $ and $xf^{\prime
}(x)/f(x)\rightarrow \beta $ as $x\rightarrow \infty $. Such a result is
also provided in Bertrand (1842, page 3) who mentions A.\ De Morgan in his
analysis.

3) In view of Duhamel (1839), Raabe's test is often called the Raabe-Duhamel
test.

\subsection{The case $\protect\alpha =-1$}

The following result may be useful in the case where $\alpha =-1$.

\begin{proposition}
Let $\alpha (n)=n(\left\vert a_{n+1}/a_{n}\right\vert -1)\rightarrow \alpha $
and define $B(n)$ as 
\[
B(n)=\sum_{k=k%
{{}^\circ}%
}^{n}\frac{\alpha (k)-\alpha }{k}\text{, }n\geq k%
{{}^\circ}%
\geq 0\text{.} 
\]

(i) Suppose that $0\leq B(n)$ for all $n\geq k%
{{}^\circ}%
$ and $k%
{{}^\circ}%
$ sufficiently large. Then $\lim \inf_{n\rightarrow \infty }n^{-\alpha
}a_{n}>0$. Hence if $\alpha \geq -1$, we have $\sum_{k=1}^{n}\left\vert
a_{k}\right\vert \rightarrow \infty $.

(ii) The same conclusion holds when $B(n)\rightarrow L$, a finite limit.
\end{proposition}

\begin{proof}
Using (4) we have $\alpha (n)\rightarrow \alpha $ and we have%
\[
\left\vert a_{n+1}\right\vert =\left\vert a_{n}\right\vert (1+\frac{\alpha
(n)}{n})\text{.}
\]%
For fixed $k%
{{}^\circ}%
$ it follows that $\left\vert a_{n+1}\right\vert =C\Pi _{k=k%
{{}^\circ}%
}^{n}(1+\alpha (k)/k)$ so that%
\[
\log \left\vert a_{n+1}\right\vert =C%
{{}^\circ}%
+\sum_{k=k%
{{}^\circ}%
}^{n}\log (1+\frac{\alpha (k)}{k})
\]%
We find that 
\[
\log \left\vert a_{n+1}\right\vert =C%
{{}^\circ}%
+D(n)+B(n)+\alpha E(n)+\alpha \log n\text{,}
\]%
where 
\begin{eqnarray*}
D(n) &=&\sum_{k=k%
{{}^\circ}%
}^{n}(\log (1+\frac{\alpha (k)}{k})-\frac{\alpha (k)}{k})\text{,} \\
B(n) &=&\sum_{k=k%
{{}^\circ}%
}^{n}\frac{\alpha (k)-\alpha }{k} \\
E(n) &=&\sum_{k%
{{}^\circ}%
}^{n}\frac{1}{k}-\log n
\end{eqnarray*}%
Note that (using $\alpha (k)/k\rightarrow 0$ and $\log (1+z)-z\sim z^{2}/2$,
as $z\rightarrow 0$) we have $\lim_{n\rightarrow \infty }D(n)\rightarrow
D<\infty $. Also we have $E(n)\rightarrow E$, a finite number. It follows
that%
\[
\log n^{-\alpha }\left\vert a_{n+1}\right\vert =C%
{{}^\circ}%
+D(n)+E(n)+B(n)\text{.}
\]%
If we assume that $0\leq B(n)$ for all $n\geq k%
{{}^\circ}%
$, then we have that 
\[
\log n^{-\alpha }\left\vert a_{n+1}\right\vert \geq C%
{{}^\circ}%
+D(n)+E(n)\text{.}
\]%
and $\lim \inf \log (n^{-\alpha }\left\vert a_{n+1}\right\vert )>-\infty $.
If $B(n)\rightarrow L$, then $\log n^{-\alpha }\left\vert a_{n+1}\right\vert
\rightarrow C%
{{}^\circ}%
{{}^\circ}%
$, a finite constant. In both cases, we can conclude that $%
\sum_{k=1}^{n}\left\vert a_{k}\right\vert \rightarrow \infty $ for $\alpha
\geq -1$.
\end{proof}

\bigskip

\textbf{Examples}.

1) Let $a_{n}=\Pi _{k=1}^{n}(2-e^{1/n})$. We have 
\[
n(\frac{a_{n+1}}{a_{n}}-1)=n(1-e^{1/(n+1)})\rightarrow -1
\]%
and we find that $(a_{n})\in RS_{-1}$. Here 
\[
\alpha (n)+1=n(1-e^{1/(n+1)})+1\sim \frac{1}{2n}
\]%
and $\sum_{k%
{{}^\circ}%
}^{\infty }(\alpha (k)+1)/k\rightarrow L$. It follows that $%
\sum_{k=1}^{n}\left\vert a_{k}\right\vert \rightarrow \infty $.

2) Take $a_{n}=(-1)^{n-1}\Gamma (2n+\beta +1)!/4^{n}(n!)^{2}$. We have 
\begin{eqnarray*}
\frac{\left\vert a_{n+1}\right\vert }{\left\vert a_{n}\right\vert } &=&(1+%
\frac{\beta }{2(n+1)})(1+\frac{\beta -1}{2(n+1)}) \\
&=&1+\frac{2\beta -1}{2(n+1)}+\frac{\beta (\beta -1)}{4(n+1)^{2}}
\end{eqnarray*}

and%
\[
n(\frac{\left\vert a_{n+1}\right\vert }{\left\vert a_{n}\right\vert }%
-1)\rightarrow \beta -1/2 
\]%
Using the notations as before, we have%
\[
\alpha (n)-(\beta -1/2)=\frac{2\beta -1}{2(n+1)}+\frac{n\beta (\beta -1)}{%
4(n+1)^{2}} 
\]%
and then we have 
\[
n(\alpha (n)-(\beta -1/2))\rightarrow \frac{1}{2}-\beta +\frac{\beta (\beta
-1)}{4} 
\]%
When $\beta =-1/2$ this is 
\[
n(\alpha (n)+1)\rightarrow \frac{19}{16} 
\]%
We conclude that $(\left\vert a_{n}\right\vert )\in RS_{\beta -1/2}$. If $%
\beta <-1/2$, then $\sum_{k=1}^{\infty }\left\vert a_{k}\right\vert <\infty $%
. If $\beta \geq -1/2$, then $\sum_{k=1}^{\infty }\left\vert
a_{k}\right\vert =\infty $.

3)\ Let $a_{n}=4^{n-1}((n-1)!)^{2}/((2n-1)!!)^{2}$. We have 
\[
\alpha (n)=n(\frac{a_{n+1}}{a_{n}}-1)=n(\frac{4n^{2}}{(2n+1)^{2}}-1)=-\frac{%
4n^{2}}{(2n+1)^{2}}\rightarrow -1
\]%
Also we have 
\[
\alpha (n)+1=1-\frac{4n^{2}}{(2n+1)^{2}}=\frac{4n+1}{(2n+1)^{2}}
\]%
and $n(\alpha (n)+1)\rightarrow 1$.

\subsection{An alternating Series}

Suppose that $a_{n}>0$ and consider the sequence $p_{n}=(-1)^{n}a_{n}$. For $%
\sum_{k=1}^{n}\left\vert p_{k}\right\vert $ we can use the results of the
previous section. In this section we study $\sum_{k=1}^{n}p_{k}$ and to this
end we define $b_{n}=a_{2n}$ for $n\geq 1$ and $c_{n}=a_{2n+1}$ for $n\geq 0$%
.

We assume that $\lim_{n\rightarrow \infty }n(a_{n+1}/a_{n}-1)=\alpha $ so
that $(a_{n})\in RS_{\alpha }$. As before we define $\alpha
(n)=n(a_{n+1}/a_{n}-1)$. Clearly we have $b_{n}\sim 2^{\alpha }a_{n}$ and $%
c_{n}\sim 2^{\alpha }a_{n}$.

\subsubsection{Case 1}

If $\alpha <-1$, then $\sum_{k=1}^{\infty }(-1)^{n}a_{n}=\sum_{k=1}^{\infty
}(c_{k}-b_{k})=\sum_{k=1}^{\infty }c_{k}-\sum_{k=1}^{\infty }b_{k}$ is finite

\subsubsection{Case 2}

If $\alpha >-1$, then $\sum_{k=1}^{\infty }c_{k}=\sum_{k=1}^{\infty
}b_{k}=\infty $. Now we consider $c_{n}-b_{n}$. Clearly we have 
\[
\frac{2n(c_{n}-b_{n})}{a_{2n}}=\frac{2n(a_{2n+1}-a_{2n})}{a_{2n}}=\alpha
(2n)\rightarrow \alpha 
\]

If $\alpha \neq 0$, then we have 
\[
c_{n}-b_{n}\sim \frac{\alpha }{2n}a_{2n}\sim \alpha 2^{\alpha -1}\frac{a_{n}%
}{n}\text{.} 
\]

so that $\left\vert c_{n}-b_{n}\right\vert \in RS_{\alpha -1}$.

If $\alpha <0$, then $\sum_{k=1}^{\infty }(c_{k}-b_{k})<\infty $ and $%
\sum_{k=n}^{\infty }(c_{k}-b_{k})\sim -2^{\alpha -1}a_{n}.$

If $\alpha >0$, then $\sum_{k=1}^{n}(c_{k}-b_{k})\sim 2^{\alpha
-1}a_{n}\rightarrow \infty $.

If $\alpha =0$, we have 
\[
c_{n}-b_{n}=\alpha (2n)\frac{a_{2n}}{2n}\sim \alpha (2n)\frac{a_{n}}{n}\text{%
.} 
\]

The asymptotic behaviour of the partial sum $\sum_{k=1}^{n}(c_{k}-b_{k})$
here depends on the speed of convergence in $\alpha (n)\rightarrow 0$.

\bigskip 

\textbf{Examples}

1) If $p_{n}=(-1)^{n}/\log (n+1)$ we have $\alpha =0$, $b_{n}=1/\log (2n+1)$
and $c_{n}=1/\log (2n+2)$. Also we have 
\begin{eqnarray*}
c_{n}-b_{n} &=&\frac{1}{\log (2n+2)}-\frac{1}{\log (2n+1)} \\
&=&-\frac{\log (1+1/(2n+2))}{\log (2n+1)\log (2n+1)} \\
&\sim &\frac{-1}{n(\log n)^{2}}
\end{eqnarray*}

Here we have $\sum_{1}^{n}(c_{k}-b_{k})\rightarrow L<\infty $.

2) If $p_{n}=(-1)^{n}\log (n+1)$ we have $\alpha =0$, $b_{n}=\log (2n+1)$
and $c_{n}=\log (2n+2)$. Also we have 
\[
c_{n}-b_{n}=\log (1+\frac{1}{2n+1})\sim \frac{1}{2n}
\]%
and $\sum_{1}^{n}(c_{k}-b_{k})\rightarrow \infty $.

3) If $p_{n}=(-1)^{n-1}\Gamma (2n+\beta +1)!/4^{n}(n!)^{2}$, then we have 
\[
n(c_{n}-b_{n})/a_{n}\rightarrow \alpha 2^{\alpha -1}\text{,}
\]%
where $\alpha =\beta -1/2$. Depending on $\beta $ we have convergence or
divergence.

4) Let $p_{n}=(-1)^{n}\exp n^{\theta }$ where $\theta <0$ is a parameter. We
have $a_{n}\rightarrow 1$ and%
\[
\log \frac{a_{n+1}}{a_{n}}=(n+1)^{\theta }-n^{\theta }=n^{\theta }((1+\frac{1%
}{n})^{\theta }-1)\sim \theta n^{\theta -1}\text{.}
\]%
It follows that%
\[
\alpha (n)=n(\frac{a_{n+1}}{a_{n}}-1)\sim \theta n^{\theta }\text{.}
\]%
Clearly we have 
\[
c_{n}-b_{n}=\alpha (2n)\frac{a_{2n}}{2n}\sim \theta 2^{\theta -1}n^{\theta
-1}\text{.}
\]%
Since $\theta <0$, we have $\sum_{k=1}^{\infty }(c_{k}-b_{k})<\infty $.

\section{Other related tests}

\subsection{Functions between Karamata functions}

We use the following notation.\ We say that $f\preceq g$ if $\lim \inf
g(x)/f(x)>0$ or $\lim \sup f(x)/g(x)<\infty $. We use the notation $f\asymp
g $ if $f\preceq g$ and $g\preceq f$. We use a similar notation for
sequences.

Now suppose $(\phi _{n})\in RS_{\alpha }$ and $(\varphi _{n})\in RS_{\beta }$
are 2 fixed regularly varying sequences. The class $M(\phi ,\varphi )$ is
the class of sequences $(a_{n})$ for which we can find constants $A,B,N>0$
so that%
\[
A\phi _{n}\leq a_{n}\leq B\varphi _{n},\forall n\geq N\text{.} 
\]

In our notation it means that $\phi _{n}\preceq a_{n}\preceq \varphi _{n}$.
Clearly we should have $\alpha \leq \beta $. Also, it is clear that $\beta
<0 $ implies that $a_{n}\rightarrow 0$ while if $\alpha >0$, we have $%
a_{n}\rightarrow \infty $. By taking logarithms, we also have that%
\[
\alpha \leq \lim \inf \frac{\log a_{n}}{\log n}\leq \lim \sup \frac{\log
a_{n}}{\log n}\leq \beta \text{.} 
\]

\textbf{Examples}

1) Take $a_{n}=\exp \left[ \log n\right] $. We have $\left[ \log n\right]
\leq \log n<\left[ \log n\right] +1$ and then $n/e\leq a_{n}\leq n$.

2) Take $a_{n}=\exp \left\{ \alpha \log n+\beta \sin n\log n\right\} $,$%
n\geq e$. We have $-\log n\leq \sin n\log n\leq \log n$. If\ $\beta >0$ this
gives%
\[
\exp \left\{ \alpha \log n-\beta \log n\right\} \leq a_{n}\leq \exp \left\{
\alpha \log n+\beta \log n\right\} 
\]%
so that $n^{\alpha -\beta }\leq a_{n}\leq n^{\alpha +\beta }$.

\bigskip 

We can easily use the properties of $RS$ when we consider the partial sums
of $a_{n}$.

\begin{proposition}
Suppose that $(a_{n})\in M(\phi ,\varphi )$.

(i) If $\alpha >-1$, then $(n^{-1}\sum_{k=1}^{n}a_{k})\in M(\phi ,\varphi )$.

(ii) If $\beta <-1$, then $\sum_{k=1}^{\infty }a_{k}<\infty $ and $%
(n^{-1}\sum_{k=n}^{\infty }a_{k})\in M(\phi ,\varphi )$.
\end{proposition}

\begin{proof}
(i) Taking sums we have%
\[
A\sum_{N}^{n}\phi _{n}\leq \sum_{k=N}^{n}a_{k}\leq B\sum_{N}^{n}\varphi _{n} 
\]%
Since $\alpha >-1$ we have $\beta >-1$ and $\sum_{N}^{n}\phi _{n}\sim n\phi
_{n}/(\alpha +1)$ and $\sum_{N}^{n}\varphi _{n}\sim n\varphi _{n}/(\beta +1)$%
. We obtain that 
\[
\lim \inf \sum_{k=N}^{n}a_{k}/n\phi _{n}\geq A/(1+\alpha )>0 
\]%
and 
\[
\lim \sup \sum_{k=N}^{n}a_{k}/n\varphi _{n}\leq B/(1+\beta )<\infty \text{.} 
\]

(ii) Similar.
\end{proof}

The following proposition makes a connection with $\alpha (n)$.

\begin{proposition}
Let $\alpha (n)=n(a_{n+1}/a_{n}-1)$ and assume that $\alpha (n)\neq 0$ and $%
\alpha \leq \alpha (n)\leq \beta $ for all $n\geq N$. Then $(a_{n})\in
M(\phi _{n}=n^{\alpha },\varphi _{n}=n^{\beta })$.
\end{proposition}

\begin{proof}
We have 
\[
a_{n+1}=a_{n}(1+\frac{\alpha (n)}{n})=a_{k%
{{}^\circ}%
}\Pi _{k=k%
{{}^\circ}%
}^{n}(1+\frac{\alpha (k)}{k})\text{.}
\]%
Taking logarithms, we obtain that%
\[
\log a_{n+1}=\log a_{k%
{{}^\circ}%
}+D(n)+\sum_{k%
{{}^\circ}%
}^{n}\frac{\alpha (k)}{k}\text{,}
\]%
where 
\[
D(n)=\sum_{k%
{{}^\circ}%
}^{n}(\log (1+\frac{\alpha (k)}{k})-\frac{\alpha (k)}{k})
\]%
Using $\log (1+z)-z+O(z^{2})$ as $z\rightarrow 0$, it follows that $\lim
D(n)=L$, a finite limit. On the other hand, we have%
\[
\alpha \sum_{k%
{{}^\circ}%
}^{n}\frac{1}{k}\leq \sum_{k%
{{}^\circ}%
}^{n}\frac{\alpha (k)}{k}\leq \beta \sum_{k%
{{}^\circ}%
}^{n}\frac{1}{k}
\]%
Using $E(n)=\sum_{k%
{{}^\circ}%
}^{n}1/k-\log n$ and $\lim E(n)=E$, a finite limit, we have%
\[
\alpha E(n)+\alpha \log n\leq \sum_{k%
{{}^\circ}%
}^{n}\frac{\alpha (k)}{k}\leq \beta E(n)+\beta \log n\text{.}
\]%
It follows that%
\begin{eqnarray*}
&&\log a_{k%
{{}^\circ}%
}+D(n)+\alpha E(n)+\alpha \log n \\
&\leq &\log a_{n+1} \\
&\leq &\log a_{k%
{{}^\circ}%
}+D(n)+\beta E(n)+\beta \log n\text{.}
\end{eqnarray*}%
Hence $n^{\alpha }C^{\prime }(n)\leq a_{n+1}\leq n^{\beta }C(n)$ where $C(n)$
and $C^{\prime }(n)$ converge to a finite and positive limit. We conclude
that $(a_{n})\in M(\phi _{n}=n^{\alpha },\varphi _{n}=n^{\beta })$. This
proves the result.
\end{proof}

Note that we have $\alpha (n)=0$ i.o. if and only if $a_{n+1}=a_{n}$ i.o.

The proof shows the following result.

\begin{corollary}
Let $\alpha (n)=n(a_{n+1}/a_{n}-1)$ and assume that $\alpha (n)\neq 0$, $%
\forall n\geq N$.

(i) If $\alpha \leq \alpha (n),\forall n\geq N$, then $n^{\alpha }\preceq
a_{n}$

(ii) If $\alpha (n)\leq \beta ,\forall n\geq N$, then $a_{n}\preceq n^{\beta
}$.
\end{corollary}

In the next result we study $a_{\left[ nx\right] }/a_{n}$.

\begin{proposition}
Let $\alpha (n)=n(a_{n+1}/a_{n}-1)$ and assume that $\alpha (n)\neq 0$ and $%
\alpha \leq \alpha (n)\leq \beta $ for all $n\geq N$. Then $(a_{n})$
satisfies%
\begin{eqnarray*}
x^{\alpha } &\leq &\lim \binom{\sup }{\inf }\frac{a_{\left[ xt\right] }}{a_{%
\left[ t\right] }}\leq x^{\beta },\forall x>1\text{,} \\
x^{\beta } &\leq &\lim \binom{\sup }{\inf }\frac{a_{\left[ xt\right] }}{a_{%
\left[ t\right] }}\leq x^{\alpha },\forall x,0<x\leq 1\text{.}
\end{eqnarray*}
\end{proposition}

\begin{proof}
Using the notations as before, for $x>1$ and $t$ suffieciently large, we
have 
\[
\log \frac{a_{\left[ xt\right] }}{a_{\left[ t\right] }}=D(\left[ xt\right]
-1)-D(\left[ t\right] -1)+\sum_{\left[ t\right] }^{\left[ xt\right] -1}\frac{%
\alpha (k)}{k} 
\]%
We have $D(\left[ xt\right] -1)-D(\left[ t\right] -1)\rightarrow 0$ and 
\[
\alpha \sum_{\left[ t\right] }^{\left[ xt\right] -1}\frac{1}{k}\leq \sum_{%
\left[ t\right] }^{\left[ xt\right] -1}\frac{\alpha (k)}{k}\leq \beta \sum_{%
\left[ t\right] }^{\left[ xt\right] -1}\frac{1}{k}\text{.} 
\]%
\bigskip It follows that%
\[
\alpha \log x\leq \lim \binom{\sup }{\inf }\sum_{\left[ t\right] }^{\left[ xt%
\right] -1}\frac{\alpha (k)}{k}\leq \beta \log x\text{.} 
\]%
We conclude that 
\[
x^{\alpha }\leq \lim \binom{\sup }{\inf }\frac{a_{\left[ xt\right] }}{a_{%
\left[ t\right] }}\leq x^{\beta } 
\]

For $0<x\leq 1$, the proof is similar. This proves the result.
\end{proof}

\subsection{O-regular variation}

A sequence $(c_{n})$ of positive numbers is $O$-regularly varying if it
satisfies $c_{\left[ tx\right] }\preceq c_{\left[ t\right] },\forall x>0$.
We use the notation $(c_{n})\in ORS$. It can be proved that $(c_{n})\in ORS$
if and only if $c_{n}$ can be written as 
\[
c_{n}=d(n)\exp \sum_{k=1}^{n}k^{-1}\delta _{k}\text{,}
\]%
where $d(n)$ is bounded away from $0$ and from $\infty $, and $\delta _{n}$
is bounded. Note that $c_{n}\asymp b_{n}:=\exp \sum_{k=1}^{n}k^{-1}\delta
_{k}$, and we have%
\[
\frac{b_{n+1}}{b_{n}}-1=\exp \frac{\delta _{n+1}}{n+1}-1\sim \frac{\delta
_{n+1}}{n+1}
\]%
so that $\beta (n)$ defined by 
\[
\beta (n)=n(\frac{b_{n+1}}{b_{n}}-1)\text{,}
\]%
is bounded away from zero and infinity. As in the previous section it
follows that we can find real numbers $\alpha ,\beta $, and positive numbers 
$A,B,N$ so that $An^{\alpha }\leq b_{n}\leq Bn^{\beta },\forall n\geq N$.
The results of the previous section can be used.

\textbf{Remarks}

1) In Liflyland et al. (2011) the authors consider weak monotone sequences
defined as follows. Suppose that $a_{n}\geq 0$ and that $a_{n}\rightarrow 0$%
. The sequence is called weak monotone if there is a fixed constant $C$ so
that%
\[
a_{k}\leq Ca_{n},n\leq k\leq 2n\text{, }\forall n\geq 1\text{.}
\]%
For $1\leq x\leq 2$, we clearly have $a_{\left[ xn\right] }\leq Ca_{n}$. We
find that for $x\geq 1$ we have $a_{\left[ xn\right] }\preceq a_{n}$.

2) In Sayel A. Ali (2008) the author considers sequence so that $%
a_{2n}\asymp a_{n}$ and $a_{2n+1}\asymp a_{n}$.

\subsection{Gauss}

Gauss assumed that $\alpha (n)=n(a_{n+1}/a_{n}-1)\rightarrow p$ at a certain
rate.\ More precisely, Gauss assumed that $\alpha (n)=p+n^{1-r}B_{n}$, where 
$r>1$ and where $B_{n}$ is a bounded sequence. Clearly $(a_{n})\in RS_{p}$
and $\sum_{1}^{\infty }a_{n}<\infty $ if $p<-1$, $\sum_{1}^{\infty
}a_{n}=\infty $ if $p>-1$.

To study the case $p=-1$ we proceed as in the previous section and we find%
\[
\log a_{n+1}=\log a_{k%
{{}^\circ}%
}+D(n)+\sum_{k%
{{}^\circ}%
}^{n}\frac{\alpha (k)}{k}\text{,}
\]%
where $D(n)\rightarrow D$. We have 
\[
\log a_{n+1}-p\log n=C+D(n)+pE(n)+\sum_{k%
{{}^\circ}%
}^{n}\frac{B_{k}}{k^{r}}\text{.}
\]%
Clearly we have 
\[
\left\vert \sum_{k%
{{}^\circ}%
}^{n}\frac{B_{k}}{k^{r}}\right\vert \leq \sum_{k%
{{}^\circ}%
}^{n}\frac{\left\vert B_{k}\right\vert }{k^{r}}\text{.}
\]%
Since $B_{n}$ is bounded and $r>1$, we find that $\lim_{n\rightarrow \infty
}(\log a_{n+1}-p\log n)=F$, a finite limit. Hence we find that $a_{n}\sim
cn^{p}$ for some constant $c>0$. It is clear that\ $\sum_{1}^{\infty
}a_{n}=\infty $ if $p\geq -1$.

\subsection{Refining Bertrand's and Gauss' test}

Earlier we considered $\alpha (n)=n(\left\vert a_{n+1}/a_{n}\right\vert -1)$
and assumed that $\alpha (n)\rightarrow \alpha $. In this section we assume
that $b(n)(\alpha (n)-\alpha )$ is bounded or that $b(n)(\alpha (n)-\alpha
)\rightarrow \beta $.\ Here we assume that $0<b(n)\uparrow \infty $ and that 
$b(n)\geq 1$. As in Proposition 7, we have 
\[
\log n^{-\alpha }\left\vert a_{n+1}\right\vert =C%
{{}^\circ}%
+D(n)+\alpha E(n)+B(n)\equiv Q(n)+B(n)\text{,}
\]%
where 
\begin{eqnarray*}
D(n) &=&\sum_{k=k%
{{}^\circ}%
}^{n}(\log (1+\frac{\alpha (k)}{k})-\frac{\alpha (k)}{k})\text{,} \\
E(n) &=&\sum_{k%
{{}^\circ}%
}^{n}\frac{1}{k}-\log n\text{,} \\
B(n) &=&\sum_{k=k%
{{}^\circ}%
}^{n}\frac{\alpha (k)-\alpha }{k}=\sum_{k=k%
{{}^\circ}%
}^{n}\frac{b(k)(\alpha (k)-\alpha )}{b(k)k}\text{.}
\end{eqnarray*}%
Clearly $D(n)+E(n)\rightarrow c$, a finite limit. Now we consider $B(n)$.
The first result generalizes Gauss' test from the previous sectoin.

\begin{proposition}
Suppose that $b(n)(\alpha (n)-\alpha )$ is bounded and that $\sum_{k=k%
{{}^\circ}%
}^{\infty }1/(kb(k))<\infty $. Then $\left\vert a_{n}\right\vert \sim
Cn^{\alpha }$. The series $\sum_{k=1}^{n}\left\vert a_{k}\right\vert $
converges if and only if $\alpha <-1$ and diverges iff $\alpha \geq -1$.
\end{proposition}

\begin{proof}
In this case we have that $B(n)\rightarrow B$, a finite number. It follows
that $\log \left\vert a_{n+1}\right\vert n^{-\alpha }\rightarrow c$ and
hence also that $\left\vert a_{n}\right\vert \sim Cn^{\alpha }$.
\end{proof}

\textbf{Examples}

1)If $b(n)=n\log n$, or $b(n)=n/\log n$, we have $\sum_{k=k%
{{}^\circ}%
}^{\infty }1/(kb(k))<\infty $

2) If $b(n)=n^{r}$, $r>0$, we have $\sum_{k=k%
{{}^\circ}%
}^{n}1/(kb(k))=\sum_{k%
{{}^\circ}%
}^{n}k^{-r-1}$ which converges to a finite limit.

3) Consider $a_{n}=((2n-1)!!/(2n)!!)^{a}$. We have 
\[
\frac{a_{n+1}}{a_{n}}=(\frac{(2n+1)!!(2n)!!}{(2n+2)!!(2n-1)!!})^{a}=(\frac{%
2n+1}{2n+2})^{a}
\]%
and 
\[
\frac{a_{n+1}}{a_{n}}-1=(1-\frac{1}{2n+2})^{a}-1\sim -a\frac{1}{2n}
\]%
We find that $\alpha (n)\rightarrow -a/2$,so that $(a_{n})\in RS_{-a/2}$.
Using Raabes test, we have convergence if if $a>2$ and divergence if $a<2$.
For $a=2$, we have%
\[
\frac{a_{n+1}}{a_{n}}-1=(1-\frac{1}{2n+2})^{2}-1=\frac{1}{4(n+1)^{2}}-\frac{1%
}{n+1}
\]%
and%
\[
\alpha (n)+1=n(\frac{a_{n+1}}{a_{n}}-1)=\frac{n}{4(n+1)^{2}}+\frac{1}{n+1}
\]%
We find that $n(\alpha (n)+1)\rightarrow 5/4$. We have $b_{n}=n$ and from
the remark after the previous result, we have $\sum_{k=1}^{n}\left\vert
a_{k}\right\vert \uparrow \infty $.

In the next result we assume that $b(n)(\alpha (n)-\alpha )\rightarrow \beta 
$ and that $\sum_{k=k%
{{}^\circ}%
}^{n}1/(kb(k))\uparrow \infty $. We choose $k%
{{}^\circ}%
$ sufficiently large so that $\beta -\epsilon \leq b(k)(\alpha (n)-\alpha
)\leq \beta +\epsilon $, $k\geq k%
{{}^\circ}%
$.

\begin{proposition}
Suppose that $b(n)(\alpha (n)-\alpha )\rightarrow \beta $ and that $\sum_{k=k%
{{}^\circ}%
}^{n}1/(kb(k))\uparrow \infty $. Choose $\epsilon >0$.

(i) If $\beta >0$, we have $\left\vert a_{n}\right\vert \preceq n^{\alpha
+\beta +\epsilon }$.

(ii) If $\beta <0$, we have $n^{\alpha +\beta -\epsilon }\preceq \left\vert
a_{n}\right\vert $.

(iii) If $\beta =0$, we have $n^{\alpha -\epsilon }\preceq \left\vert
a_{n}\right\vert \preceq n^{\alpha +\epsilon }$.
\end{proposition}

\begin{proof}
We have%
\[
(\beta -\epsilon )\sum_{k=n%
{{}^\circ}%
}^{n}\frac{1}{kb(k)}\leq B(n)\leq (\beta +\epsilon )\sum_{k=n%
{{}^\circ}%
}^{n}\frac{1}{kb(k)}\text{.} 
\]

First consider the case where $\beta >0$.\ Since $b(k)\geq 1$, we have%
\[
\frac{1}{b(n)}\sum_{k=n%
{{}^\circ}%
}^{n}\frac{1}{k}\leq \sum_{k=n%
{{}^\circ}%
}^{n}\frac{1}{kb(k)}\leq \sum_{k=n%
{{}^\circ}%
}^{n}\frac{1}{k} 
\]%
and it follows that%
\[
(1-\epsilon )\frac{b(n)}{\log n}\leq \sum_{k=n%
{{}^\circ}%
}^{n}\frac{1}{kb(k)}\leq (1+\epsilon )\log n 
\]
\[
(\beta -\epsilon )(1-\epsilon )\frac{b(n)}{\log n}\leq B(n)\leq (\beta
+\epsilon )(1+\epsilon )\log n 
\]%
Going back to the sequence $\left\vert a_{n}\right\vert $, we find that 
\[
Q(n)+(\beta -\epsilon )(1-\epsilon )\frac{b(n)}{\log n}\leq \log \left\vert
a_{n+1}\right\vert n^{-\alpha }\leq Q(n)+(\beta +\epsilon )(1+\epsilon )\log
n. 
\]
We conclude that for any $\epsilon >0$ we have $\left\vert a_{n}\right\vert
\preceq n^{\alpha +\beta +\epsilon }$.

Next we consider the case where $\beta <0$. Now we have%
\[
(\beta -\epsilon )(1-\epsilon )\frac{b(n)}{\log n}\leq B(n)\geq (\beta
+\epsilon )(1+\epsilon )\log n 
\]%
and%
\[
Q(n)+(\beta -\epsilon )(1-\epsilon )\frac{b(n)}{\log n}\geq \log \left\vert
a_{n+1}\right\vert n^{-\alpha }\geq Q(n)+(\beta +\epsilon )(1+\epsilon )\log
n\text{.} 
\]%
We conclude that for any $\epsilon >0$ we have $\left\vert a_{n}\right\vert
\succeq n^{\alpha +\beta -\epsilon }$.

If $\beta =0$, in a similar way we find that for any $\epsilon >0$ we have $%
n^{\alpha -\epsilon }\preceq \left\vert a_{n}\right\vert \preceq n^{\alpha
+\epsilon }$. This proves the result.
\end{proof}

\begin{corollary}
Under the conditions of the Proposition we have

(i) If $\beta \geq 0$ and $\alpha +\beta <-1$, then $\sum_{1}^{\infty
}\left\vert a_{n}\right\vert <\infty $.

(ii) If $\beta \leq 0$ and $\alpha +\beta >-1$, then $\sum_{1}^{n}\left\vert
a_{n}\right\vert \uparrow \infty $.
\end{corollary}

\textbf{Examples}

1) If $b(n)=\log \log n$, we have $\sum_{k=n%
{{}^\circ}%
}^{n}1/(kb(k))\sim \log n/\log \log n$

2) If $b(n)=f(\log n)$, and $f(.)\in RV_{\theta },0\leq \theta <1$ and $%
f(x)\uparrow \infty $, then we have $\sum_{k=n%
{{}^\circ}%
}^{n}1/(kb(k))\sim \log n/((1-\theta )b(n))$

\bigskip

The proof of the previous result shows that $b(n)=\log n$ can lead to
interesting results. The next result is known as the test of Bertrand (1842).

\begin{proposition}
Assume that $b(n)(\alpha (n)-\alpha )\rightarrow \beta $ and that $b(n)=\log
n$. Then we have:

(i) If $\alpha <-1$, then $\sum_{1}^{\infty }\left\vert a_{n}\right\vert
<\infty $ and $n^{\alpha +1}(\log n)^{\beta -\epsilon }\preceq
\sum_{n}^{\infty }\left\vert a_{n}\right\vert \preceq n^{\alpha +1}(\log
n)^{\beta +\epsilon }$.

(ii) if $\alpha >-1$, then $\sum_{1}^{\infty }\left\vert a_{n}\right\vert
=\infty $ and $n^{\alpha +1}(\log n)^{\beta -\epsilon }\preceq
\sum_{1}^{n}\left\vert a_{n}\right\vert \preceq n^{\alpha +1}(\log n)^{\beta
+\epsilon }$.

(iii) if $\alpha =-1$ and $\beta >-1$, then $\sum_{1}^{\infty }\left\vert
a_{n}\right\vert =\infty $ and $(\log n)^{\beta +1-\epsilon }\preceq
\sum_{1}^{n}\left\vert a_{n}\right\vert \preceq (\log n)^{\beta +1+\epsilon
} $.

(iv) if $\alpha =-1$ and $\beta <-1$, then $\sum_{1}^{\infty }\left\vert
a_{n}\right\vert <\infty $ and $(\log n)^{\beta +1-\epsilon }\preceq
\sum_{n}^{\infty }\left\vert a_{n}\right\vert \preceq (\log n)^{\beta
+1+\epsilon }$.

(v) if $\alpha =-1$ and $\beta =-1$, then we can not come to a conclusion.
\end{proposition}

\begin{proof}
In this case we have $\sum_{k=n%
{{}^\circ}%
}^{n}1/(kb(k))\sim \log \log n$ and we find that
\end{proof}

\[
(\beta -\epsilon )\log \log n\leq B(n)\leq (\beta +\epsilon )\log \log n 
\]%
It follows that%
\[
Q(n)+(\beta -\epsilon )\log \log n\leq \log \left\vert a_{n+1}\right\vert
n^{-\alpha }\leq Q(n)+(\beta +\epsilon )\log \log n\text{,} 
\]%
and $W(n)n^{\alpha }(\log n)^{\beta -\epsilon }\leq \left\vert
a_{n+1}\right\vert \leq W^{\prime }(n)n^{\alpha }(\log n)^{\beta +\epsilon }$%
, where $W(n)$ and $W^{\prime }(n)$ converge to finite and positive limits.
The five cases of the proposition now follow.

In the case where $\alpha =\beta =-1$, the test does not lead to a
conclusion. Extensions that take care of this case go back to Martin (1941),
see also Rajagopal (1952).

\bigskip

\section{References}

\begin{enumerate}
\item N.H.\ Bingham, C.M.\ Goldie and J.L. Teugels.\ Regular Variation.
Cambridge University Press 1987.

\item J.\ Bertrand. R\`{e}gles sur la convergence des s\'{e}ries.\ Journal
de Math. (1), 7, 35 - 54, 1842.

\item R. Bojanic and E.\ Seneta. A unified theory of regularly varying
sequences. Math.\ Zeitschrift 134, 91 - 106, 1973

\item M.\ Cadena, M.\ Kratz and E.\ Omey. On functions bounded by Karamata
functions.\ To appear, 2018??

\item F. Cajori, Evolution of criteria of convergence, Bull. New York Math.
Soc. 2 1892.

\item M.\ Duhamel. Nouvelle r\`{e}gle pour la convergence des s\'{e}ries. J.
de math\'{e}matiques pures et appliqu\'{e}es, s\'{e}rie 1, tome 4, p.
214-221, 1839.

\item J.L.\ Geluk and L. de Haan. Regular variation, extensions and
Tauberian theorems. CWI Tract 40, Amsterdam, 1987.

\item W.\ Feller. An introduction to Probability Theory and Its
Applications.\ Vol.\ II, 2nd Edition Wiley, NewYork, 1971.

\item C.N.B. Hammond. The case for Raabe's test. Preprint available on
ArXiv:1801.07584v1, 2018.

\item J. Karamata. Sur un mode de croissance r\'{e}guli\`{e}re des
fonctions.\ Mathematica (Cluj) 4, 38 - 53, 1930.

\item J.\ Karamata. Sur un mode de croissance r\'{e}guli\`{e}re.\ Th\'{e}or%
\`{e}mes fondamontaux.\ Bull.\ Soc.\ Math.\ France 61, 55 62, 1933.

\item E. Liflyand, S.\ Tikhonov and M.Zeltser. Extending tests for
convergence of number series. J.\ Math.An.and Appl. 377, 194 - 206, 2011

\item M.\ Martin, A sequence of limit tests for the convergence of series.
Bull. Amer. Math. Soc. 47, 452-457, 1941.

\item F. Prus-Wisniowski. Comparison of Raabe's and Schl\"{o}mlich's tests.
Tatra Mt. Math. Publ. 42, 119 - 130, 2009.

\item J.L. Raabe, Untersuchungen uber die Convergenz und Divergenz der
Reihen, Zeitschrift f\"{u}r Physik und Mathematik 10 (1832), 41-74.

\item J.L.\ Raabe, Note zur Theorie der Convergenz and Divergenz der Reihen.
Journal f\"{u}r die reine und angewandte Mathematik,309 - 310, 1834.

\item C. T. Rajagopal, Sui criteri del rapporto per la convergenza delle
serie a termini positivi. Bollettino dell'Unione Matematica Italiana, Serie
3, Vol. 7 (4), 382 - 387, 1952.

\item Sayel A. Ali. The $m$-th Ratio Test: New Convergence Tests for Series.
American Math.\ Monthly 115, 514 - 523, 2008.

\item E.\ Seneta.\ Functions of regular variation.\ Lecture Notes
Mathematics,Vol. 506, 1976.
\end{enumerate}

\end{document}